\documentclass[12pt,english,reqno]{amsart}
\usepackage[latin9]{inputenc}
\usepackage{verbatim}
\usepackage{textcomp}
\usepackage{amsthm}
\usepackage{amstext}
\usepackage{amssymb}
\usepackage{graphicx}
\usepackage{esint}
\date{\today}

\makeatletter
\newcounter{cprop}[section]

\usepackage{amscd}\usepackage{amsthm}\usepackage[english]{babel}
\usepackage[arrow,matrix]{xy}\usepackage{cite}\usepackage{bbm}
\@ifundefined{definecolor}
 {\usepackage{color}}{}
\usepackage{MnSymbol}

\theoremstyle{theorem}
\newtheorem{thm}[cprop]{Theorem}
\newtheorem{lemma}[cprop]{Lemma}

\newtheorem{cor}[cprop]{Corollary}
\theoremstyle{definition}
\newtheorem{rem}[cprop]{Remark}

\newtheorem{rk}[cprop]{Remark}
\newtheorem{defn}[cprop]{Definition}

\numberwithin{equation}{section}

\newcommand{\F}{{\mathcal F}}

\newcommand{\dd}{{\mathrm{d}}}

\newcommand{\LL}{{\mathcal L}}
\newcommand{\EE}{{\mathcal E}}
\newcommand{\E}{{\mathbb E}}

\newcommand{\w}{{\bf w}}
\newcommand{\s}{{\sigma}}

\newcommand{\bea}{\begin{eqnarray}}
\newcommand{\eea}{\end{eqnarray}}

\newcommand{\N}{\mathbb N}

\newcommand{\R}{\mathbb{R}}

\newcommand{\PP}{\mathbb{P}}
\newcommand{\FF}{\mathbb{F}}

\def\d {\triangle}
\def\t {\theta}
\def\w {\omega}
\def\> {\Rightarrow}
\def\0 {\emptyset}

\def\a {\alpha}
\def\l {\lambda}
\def\d {\delta}
\def\e {\varepsilon}

\def\b {\beta}
\def\g {\gamma}
\def\s {\sigma}

\newcommand{\diam}{\mathrm{diam}}

\makeatother

\usepackage{babel}

\begin{document}
\title[Weak synchronization for isotropic flows]{Weak synchronization for isotropic flows}

\begin{abstract}
We study Brownian flows on manifolds for which the associated Markov process is strongly mixing with respect to an invariant probability measure 
and for which the distance process for each pair of trajectories is a diffusion $r$. We provide a sufficient condition on the boundary behavior of $r$ at $0$ 
which guarantees that the statistical equilibrium of the flow is almost surely a singleton and its support is a weak point attractor. The condition is 
fulfilled in the case of negative top Lyapunov exponent, but it is also fulfilled in some cases when the top Lyapunov exponent is zero. Particular examples are 
isotropic Brownian flows on $S^{d-1}$ as well as isotropic Ornstein-Uhlenbeck flows on $\R^d$. 
\end{abstract}

\author[M. Cranston]{Michael Cranston}
\address{Department of Mathematics\\
University of California, Irvine, USA}
\email{mcransto@math.uci.edu}

\author[B. Gess]{Benjamin Gess}
\address{Max-Planck Institute for Mathematics in the Sciences \\
04103 Leipzig\\
Germany }
\email{bgess@mis.mpg.de}

\author[M. Scheutzow]{Michael Scheutzow}
\address{Institut f\"ur Mathematik, MA 7-5\\
Technische Universit\"at Berlin\\
10623 Berlin \\
Germany}
\email{ms@math.tu-berlin.de}



\maketitle

\section{Introduction}
 
We study the asymptotic behavior of white-noise random dynamical systems\footnote{For notation and some background on random dynamical systems and random attractors see Section \ref{sec:not} below.} (RDS) $\varphi$ on complete, seperable metric spaces $(E,d)$. 
We present necessary and sufficient conditions for weak synchronization, which means that there is a weak point attractor for $\varphi$ consisting of a single random point. In particular, in this case
    $$d(\varphi_{t}(\cdot,x),\varphi_{t}(\cdot,y)) \to 0 \text{ for }t\to\infty,$$
in probability, for all $x,y\in E$.

In \cite{FGS14} (weak) synchronization for RDS generated by SDE driven by additive noise
\begin{equation}\label{eq:SDE}
 \dd X_t = b(X_t)\,\dd t + \dd W_t
\end{equation}
has been analyzed and general necessary and sufficient conditions for (weak) synchronization have been derived, 
based on a local asymptotic stability condition. More precisely, the existence of a non-empty open set $U$ and a sequence $t_n \uparrow \infty$ such that $U$ is contracted along the flow, that is,
  $$\PP(\lim_{n\to\infty}\diam\ \varphi_{t_n}(\cdot,U)=0)>0$$
was assumed in \cite{FGS14}. This condition was shown to be satisfied in the case that the top Lyapunov exponent is negative $\l_1 < 0$ using the stable manifold theorem. In addition, in \cite{FGS14} weak synchronization has been shown for \eqref{eq:SDE} in the gradient case, i.e.\ if $b=\nabla V$ for some $V\in C^2$ among further assumptions.

These results are complemented by the present work which concentrates on the case of vanishing top Lyapunov exponent $\l_1 = 0$. In this case, local asymptotic stability as used in \cite{FGS14} is not satisfied in general. As a first main result we present necessary and sufficient conditions for weak synchronization. More precisely, in Theorem \ref{main} we show that weak synchronization holds if and only if $\varphi$ is strongly mixing and satisfies a global weak pointwise stability condition (cf.\ Definition \ref{def:ptw:stab}). In a sense, this condition replaces the local asymptotic stability condition required in \cite{FGS14}. As it turns out, this global weak pointwise stability condition is particularly easy to check in the case of the distance $r_t:=d(\varphi_t(x),\varphi_t(y))$ being a diffusion. In this case it follows from the speed measure of 
$r_t$ being infinite.

This general result is then used in order to prove weak synchronization for isotropic Brownian flows on the sphere satisfying $\l_1 \le 0$, as well as for isotropic 
Ornstein-Uhlenbeck flows with $\l_1 \le 0$. As detailed above, the cases of vanishing top Lyapunov exponent $\l_1 = 0$ could not be treated by previous methods.

\subsection{Preliminaries and notation}\label{sec:not}
We start by fixing our set-up which is the same as in \cite{FGS14}.
Let $(E,d)$ be a complete separable metric space with Borel $\sigma$-algebra $\EE$ and $\left(  \Omega,\F,\PP, \t\right)  $ be an {\em ergodic metric dynamical system}, that is, $(\Omega,\F,\PP)$ is a probability space (not necessarily complete) and  $\t:=\left(  \theta_{t}\right)  _{t\in\R}$ 
is a group of jointly measurable maps on $\left(  \Omega,\F,\PP\right)$ with ergodic invariant measure $\PP$.

Further, let $\varphi: \R_+ \times \Omega \times E \rightarrow E$ be a {\em perfect cocycle}, that is, $\varphi$ is measurable with $\varphi_{0}(  \omega,x)  =x$ and $\varphi_{t+s}\left(
\omega,x\right)  =\varphi_{t}\left(  \theta_{s}\omega,\varphi_{s}\left(
\omega,x\right)  \right)  $ for all $x\in E$, $t,s\geq0$, $\omega\in\Omega$. 
We will assume continuity of the map $x\mapsto \varphi_{s}(\w,x)$ for each $s \ge 0$ and $\omega \in \Omega$. 
The collection $(\Omega, \F,\PP,\t,\varphi)$ is then called a 
{\em random dynamical system} (in short: RDS), see \cite{A98} for a comprehensive treatment. 

Since our main applications are RDS generated by SDE driven by Brownian motion, we will assume that the RDS $\varphi$ is suitably 
adapted to a filtration and is of white noise type in the following sense: We will assume that there exists a family $\FF=(\F_{s,t})_{-\infty < s \le t < \infty}$ of sub$-\sigma$ algebras of $\F$ such that $\F_{t,u}\subseteq \F_{s,v}$ whenever $s \le t \le u \le v$, $\theta_r^{-1}(\F_{s,t})=\F_{s+r,t+r}$ for all $r,s,t$ and $\F_{t_0,t_1},\cdots,\F_{t_{n-1},t_n}$ are independent for all $t_0 \le \cdots\le t_n$. For each $t \in \R$, $\F_t$ denotes the smallest $\sigma$-algebra containing all $\F_{s,t}$, $s \le t$ and $\F_{t,\infty}$ denotes the smallest $\sigma$-algebra containing all $\F_{t,u}$, $t \le u$. Note that for each $t \in \R$, the $\sigma$-algebras 
$\F_t$ and $\F_{t,\infty}$ are independent. We will assume that $\varphi_{s}(\cdot,x)$ is $\F_{0,s}$-measurable for each $s \ge 0$. The collection 
$(\Omega, \F,\FF,\PP,\t,\varphi)$ (or just $\varphi$) is then called a {\em white noise (filtered) RDS}.

If $\varphi$ is a white noise RDS then we  define the associated Markovian (Feller) semigroup by
  $$ P_t f(x):=\E f(\varphi_t(\cdot,x)), $$
for $f$ measurable and bounded. If there exists an invariant probability measure $\rho$ for $(P_t)$, then for every sequence $t_k \to \infty$ the weak$^*$ limit, the so-called {\em statistical equilibrium} of $\varphi$,
\begin{equation}\label{eq:inv_meas}
   \mu_\omega := \lim_{k\to\infty}\varphi_{t_k}(\t_{-{t_k}}\omega)\rho 
\end{equation}
exists $\PP$-a.s.\ and does not depend on the sequence $t_k$, $\PP$-a.s.
The measure $\mu_\omega$ can be chosen to be $\F_0$-measurable and it satisfies, for each $t\ge 0$, $\mu_\omega \varphi_t^{-1}(\omega,.)=\mu_{\theta_t\omega}$  $\PP$-a.s.\ and 
$\E\mu_\omega=\rho$. 

A Markovian semigroup $(P_t)$ with invariant measure $\rho$ is said to be \textit{strongly mixing} if 
 $$P_t f(x) \to \int_E f(y) \dd \rho(y)\quad \text{for } t\to\infty$$ 
for each continuous, bounded $f$ and all $x\in E$. Similarly, an RDS $\varphi$ is said to be strongly mixing if the law of $\varphi_t(\cdot,x)$ converges to $\rho$ for 
$t\to\infty$ for all $x\in E$.

\begin{defn}
Let $(\Omega, \F,\PP,\t,\varphi)$ be a white-noise RDS. We say that {\em weak synchronization} occurs if there exists an $\F_0$-measurable random variable $a$ such that
\begin{enumerate}
\item for all $t \ge 0$, $\varphi_t(a(\omega))=a(\theta_t\omega)$ almost surely, and 
\item for every $x \in E$, we have
$$
\lim_{t \to \infty} d(\varphi_t(\theta_{-t}\omega,x),a(\omega))=0, \mbox{ in probability}.
$$
\end{enumerate}
We say that {\em synchronization} occurs if (1) holds and (2) is replaced by
\begin{itemize}
 \item[(2')] for every compact set $B \subset E$ we have 
$$
\lim_{t \to \infty} \sup_{x \in B}d(\varphi_t(\theta_{-t}\omega,x),a(\omega))=0, \mbox{ in probability}.
$$
\end{itemize}
\end{defn}

Note that (weak) synchronization means that there exists a weak (point) attractor which is a singleton (see \cite{FGS14} for further details).

\section{Main result}

\subsection{RDS on metric spaces}
Let $(E,d)$ be a complete, separable metric space and assume that $\varphi$ is an $E$-valued white noise random dynamical system for which the associated 
Markov process has an ergodic invariant probability measure $\rho$.

\begin{defn}\label{def:ptw:stab} The RDS $\varphi$ is said to satisfy
 \begin{enumerate}
  \item weak local pointwise stability iff there is a set $U\subseteq E$ with $\rho(U)>0$ and a sequence $t_n \uparrow \infty$ such that for all $x,y\in U$ there is a 
  $\d=\d(x,y)>0$ such that for all $\eta>0$:
  \begin{equation}\label{eq:ptw_stab}
     \liminf_{n\to\infty}\PP\big(d(\varphi_{t_n}(\cdot,x),\varphi_{t_n}(\cdot,y))\le \eta\big) \ge \d(x,y) >0.
  \end{equation}
  \item weak global pointwise stability iff $U$ in (1) can be chosen such that $\rho(U)=1$.
\end{enumerate}
\end{defn}

We first present a slight generalization of \cite[Lemma 2.19 (ii)]{FGS14} (which in turn is based on the main result in \cite{LJ87}).
\begin{lemma}\label{lem:structure_stat_eq}
Assume that $\varphi$ satisfies weak local pointwise stability. Then the statistical equilibrium is discrete, that is  $\mu_\omega$ consists of finitely many atoms of the same mass $\mathbb{P}$-a.s., i.e.\ there is an $N\in \N$ and $\mathcal{F}_0$-measurable random variables $a_1,\dots,a_N$ such that
         $$ \mu_\omega= \Big\{\frac{1}{N}\d_{a_i(\omega)}:i=1,\dots,N\Big\}.$$
\end{lemma}
\begin{proof}
   Due to \cite[Lemma 2.19 (i)]{FGS14}, the statistical equilibrium $\mu_\omega$ is either discrete or diffuse, which means that  $\mu_\omega$ does not have point masses $\PP$-a.s.. %
 
   Hence, we only have to show that $\mu_\omega$ has a point mass with positive probability. Denote the diagonal in $E\times E$ by 
   $\Delta$. Then there exists a measurable 
   function $\psi:(E\times E)\setminus\Delta \to [0,\infty)$ such that $\psi(x,y)\to\infty$ for $d(x,y)\to 0$ and 
       $$ \E\int_{(E\times E)\setminus\Delta} \psi(x,y)\dd\mu_\omega(x)\dd\mu_\omega(y)<\infty.$$
   By invariance of $\mu_\omega$ we get
     \begin{equation}\begin{split}\label{eq:LJ}
       &\E\int_{(E\times E)\setminus\Delta} \psi(x,y)\dd\mu_\omega(x)\dd\mu_\omega(y)\\
       &\ge \E\int_{(E\times E)\setminus\Delta}\psi(\varphi_t(\omega,x),\varphi_t(\omega,y))\dd\mu_\omega(x)\dd\mu_\omega(y)\\
       &\ge \E\int_{(E\times E)\setminus\Delta}1_{U}(x)1_{U}(y)\psi(\varphi_t(\omega,x),\varphi_t(\omega,y))\dd\mu_\omega(x)\dd\mu_\omega(y).
     \end{split}\end{equation}
     By assumption there is a sequence $t_n \to \infty$ such that, for all $x,y\in U$, $\eta>0$
       \begin{equation}\label{delta} 
        \liminf_{n\to\infty}\PP(d(\varphi_{t_n}(\cdot,x),\varphi_{t_n}(\cdot,y))\le \eta) \ge \d(x,y)>0. 
      \end{equation}
      We can and will assume that $\delta$ depends measurably upon $(x,y)$ (e.g. by 
     defining $\delta(x,y)$ as the $\lim_{\eta \to 0}$ of the left hand side of \eqref{delta}). 
     We define $C(n,x,y,R):=\{\omega\in \Omega: \psi(\varphi_{t_n}(\omega,x),\varphi_{t_n}(\omega,y))\ge R\}$ and observe that
          $$\liminf_{n\to\infty} \PP(C(n,x,y,R)) \ge \d(x,y) > 0, $$
     {\color{black}for all  $x,y\in U$ and all $R>0$.} 
From \eqref{eq:LJ} we obtain that
        \begin{equation*}\begin{split}
           &\E\int_{(E\times E)\setminus\Delta} \psi(x,y)\dd\mu_\omega(x)\dd\mu_\omega(y)\\
           &\ge R \E\int_{(E\times E)\setminus\Delta}1_{U}(x)1_{U}(y)1_{C(n,x,y,R)}(\omega)\dd\mu_\omega(x)\dd\mu_\omega(y).
        \end{split}\end{equation*}
        Using that $\mu_\omega$ is $\F_{0}$-measurable, $C(n,x,y,R)$ is $\F_{0,\infty}$-measurable and that $\F_{0}$, $\F_{0,\infty}$ are independent, we conclude that
        \begin{equation*}\begin{split}
           &\E\int_{(E\times E)\setminus\Delta}1_{U}(x)1_{U}(y)1_{C(n,x,y,R)}(\omega)\dd \mu_\omega(x)\dd \mu_\omega(y)\\
           &=\E\E\left[\int_{(E\times E)\setminus\Delta}1_{U}(x)1_{U}(y)1_{C(n,x,y,R)}(\omega)\dd \mu_\omega(x)\dd \mu_\omega(y)\Big|\F_0\right]\\
           &=\E\tilde\E\int_{(E\times E)\setminus\Delta}1_{U}(x)1_{U}(y)1_{C(n,x,y,R)}(\tilde\omega)\dd \mu_\omega(x)\dd \mu_\omega(y)\\
           &=\E\int_{(E\times E)\setminus\Delta}1_{U}(x)1_{U}(y)\PP[C(n,x,y,R)]\dd \mu_\omega(x)\dd \mu_\omega(y).
        \end{split}\end{equation*}
        Using this above, taking $\liminf_{n\to\infty}$ and using Fatou's lemma yields
        \begin{equation*}\begin{split}
           &\E\int_{(E\times E)\setminus\Delta} \psi(x,y)\dd\mu_\omega(x)\dd\mu_\omega(y)\\
           &\ge R \E\int_{(E\times E)\setminus\Delta}1_{U}(x)1_{U}(y)\liminf_{n\to\infty}\PP[C(n,x,y,R)]\dd\mu_\omega(x)\dd\mu_\omega(y)\\
           &\ge R \E\int_{(E\times E)\setminus\Delta}1_{U}(x)1_{U}(y)\d(x,y)\dd \mu_\omega(x)\dd\mu_\omega(y).
        \end{split}\end{equation*}
        If $\mu_\omega$ has no point masses, then $(\mu_\omega\otimes \mu_\omega)(\Delta)=0$ and thus	
        \begin{align*}
          \E\int_{(E\times E)\setminus\Delta}1_{U}(x)1_{U}(y)\d(x,y)\dd\mu_\omega(y)\dd\mu_\omega(x)
          &=\E\int_{E\times E}1_{U}(x)1_{U}(y)\d(x,y)\dd\mu_\omega(y)\dd\mu_\omega(x).
        \end{align*}
        Assume that the right hand side is $0$. Then, for a.a.\ $\omega\in \Omega$ and ${\mu_\omega}_{|U}$-a.a.\ $x\in E$ we have
          $$ \int_E 1_{U}(y)\d(x,y)\dd\mu_\omega(y) = 0, $$
        which implies $\mu_\omega(U)=0$. Since $\rho(U)=\E\mu_\omega(U)>0$ this implies $\rho(U)=0$, a contradiction. Thus,
        $$\E\int_{E\times E}1_{U}(x)1_{U}(y)\d(x,y)\dd\mu_\omega(y)\dd\mu_\omega(x)>0.$$
        Since $R>0$ is arbitrary we obtain a contradiction. This concludes the proof.
\end{proof}

\begin{lemma}\label{lem:very_weak_sync}
Assume that $\varphi$ satisfies weak global pointwise stability. Then the statistical equilibrium is supported by a single point.
\end{lemma}
\begin{proof}
  By Lemma \ref{lem:structure_stat_eq} the statistical equilibrium is discrete, with
  \begin{equation}\label{eq:discr_stat_eq}
    \mu_\omega= \Big\{\frac{1}{N}\d_{a_i(\omega)}:i=1,\dots,N\Big\}.
  \end{equation}
  Let
    $$A(\w):=\{a_i(\omega):\ i=1,\dots,N\}.$$
  \textit{Step 1:} We show that for all $i,j\in \{1,\dots,N\}$ there is a $\d_{i,j}>0$ such that for every $\eta>0$, we have
    \begin{equation}\label{eqn:pos_prob} 
  	\liminf_{n\to\infty}\PP(d(\varphi_{t_n}(\cdot,a_i(\cdot)),\varphi_{t_n}(\cdot,a_j(\cdot))\le \eta)\ge \d_{i,j}.
    \end{equation}
   By assumption, there is a set $U\subseteq E$ with $\rho(U)=1$ such that \eqref{eq:ptw_stab} holds for all $x,y\in U$. Due to \eqref{eq:discr_stat_eq} we have $a_i(\omega) \in U$ for $\PP$-a.a.\ $\omega\in\Omega$ and all $i=1,\dots,N$.
   
   Since $\varphi$ is a white noise RDS we have that
   \begin{align*}
     \PP(d(\varphi_{t_n}(\cdot,a_i(\cdot)),\varphi_{t_n}(\cdot,a_j(\cdot)))\le \eta)
     &= \E(1_{d(\varphi_{t_n}(\cdot,a_i(\cdot)),\varphi_{t_n}(\cdot,a_j(\cdot)))\le \eta})\\
     &= \E\E(1_{d(\varphi_{t_n}(\cdot,a_i(\cdot)),\varphi_{t_n}(\cdot,a_j(\cdot)))\le \eta}|\F_0) \\
     &= \E\tilde\E(1_{d(\varphi_{t_n}(\omega,a_i(\tilde\omega)),\varphi_{t_n}(\w,a_j(\tilde\w)))\le \eta})\\
     &= \tilde\E\PP(d(\varphi_{t_n}(\omega,a_i(\tilde\omega)),\varphi_{t_n}(\w,a_j(\tilde\w)))\le \eta).
    \end{align*}
    Thus,
   \begin{align*}
     &\liminf_{n\to\infty}\PP(d(\varphi_{t_n}(\cdot,a_i(\cdot)),\varphi_{t_n}(\cdot,a_j(\cdot))\le \eta)\\
     &= \liminf_{n\to\infty}\tilde\E\PP(d(\varphi_{t_n}(\omega,a_i(\tilde\omega)),\varphi_{t_n}(\w,a_j(\tilde\w))\le \eta)\\
     &\ge \tilde\E\liminf_{n\to\infty}\PP(d(\varphi_{t_n}(\omega,a_i(\tilde\omega)),\varphi_{t_n}(\w,a_j(\tilde\w))\le \eta)\\
     &\ge \tilde\E\d(a_i(\tilde\omega),a_j(\tilde\omega))\\
     &=:\d_{i,j}>0.
    \end{align*}
    
	\textit{Step 2:} Assume that $A(\omega)$ is not a singleton $\PP$-a.s.. Then
	\begin{equation}\label{eq:F_pos}
	  F(\omega):=\min_{i,j=1,\dots,N,\ i\ne j}d(a_i(\omega),a_j(\omega)) >0,\quad \PP\text{-a.s.}
	\end{equation}
	Moreover, since $\varphi_t(\omega,A(\omega))=A(\t_t\omega)$ we get
	\begin{align*}
	    F(\t_t\omega)
	    &=\min_{i,j=1,\dots,N,\ i\ne j}d(a_i(\t_t\omega),a_j(\t_t\omega))\\
	    &=\min_{i,j=1,\dots,N,\ i\ne j}d(\varphi_t(\omega,a_i(\omega)),\varphi_t(\omega,a_j(\omega)))\\
	    &\le d(\varphi_t(\omega,a_1(\omega)),\varphi_t(\omega,a_2(\omega))).
	\end{align*}
    Hence, for all $\eta>0$
	\begin{align*}
	    \PP(F(\cdot)\le \eta)
	    &= \PP(F(\t_{t_n}\cdot)\le \eta) \\
	    & \ge \PP\big(d(\varphi_{t_n}(\cdot,a_1(\cdot)),\varphi_{ t_n}(\cdot,a_2(\cdot)))\le \eta\big).
	\end{align*}
	Taking $\liminf_{n\to\infty}$ and using \eqref{eqn:pos_prob} we conclude that
	\begin{align}\label{eqn:stopped_pos_prob}
	            \PP(F(\cdot)\le \eta) \ge \d_{1,2} > 0,
	\end{align}
	for all $\eta>0$ in contradiction to \eqref{eq:F_pos}.
\end{proof}

\begin{thm}\label{main} 
Let $(\varphi,\theta)$ be a white noise $E$-valued RDS which is strongly mixing with respect to an invariant probability measure $\rho$ and satisfies weak global 
pointwise stability. Then weak synchronization holds for $\varphi$.
\end{thm}
\begin{proof}
  By Lemma \ref{lem:very_weak_sync}, the statistical equilibrium $\mu_\omega$ is supported by a single random point. Since $\varphi$ is strongly mixing, by \cite[Proposition 2.20]{FGS14} this implies weak synchronization.
\end{proof}

\begin{lemma}\label{lem:strong_mix}
  Let $(\varphi,\theta)$ be a white noise $E$-valued RDS with invariant probability measure $\rho$. Assume that the distance of each pair of points converges to $0$ in probability as 
  $t \to \infty$. Then $\varphi$ is strongly mixing.
\end{lemma}
\begin{proof} 
  Denote the law of a random variable by $\LL(.)$. For $f\in C_b(E)$ and $x \in E$ we have
  \begin{align*}
     |\LL(\varphi_t(\cdot,x))(f)-\rho(f)| 
     &=|\LL(\varphi_t(\cdot,x))(f)-P_t^*\rho(f)| \\
     &=|\E f(\varphi_t(\cdot,x)) - \int_E \E f(\varphi_t(\cdot,y))\dd\rho(y)| \\
     &\le \int_E \E| f(\varphi_t(\cdot,x)) - f(\varphi_t(\cdot,y))|\dd\rho(y).
  \end{align*}
  Since $d(\varphi_t(\cdot,x),\varphi_t(\cdot,y))\to 0$ in probability for all $x,y \in E$, by dominated convergence, we have $\E| f(\varphi_t(\cdot,x)) - f(\varphi_t(\cdot,y))|\to 0$. Again by dominated convergence we conclude
    \begin{align*}
       |\LL(\varphi_t(\cdot,x))(f)-\rho(f)| 
       \to 0,
    \end{align*}
  for $t\to \infty$.
\end{proof}

\begin{cor}\label{cor:main} 
Let $(\varphi,\theta)$ be a white noise $E$-valued RDS with invariant probability measure $\rho$, such that for all $x,y\in E$ 
    $$d(\varphi_{t}(\cdot,x),\varphi_{t}(\cdot,y)) \to 0 \text{ for }t\to\infty,$$
 in probability. Then weak synchronization holds for $\varphi$.
\end{cor}

\begin{rem} The following converse statement holds true.
   Assume that $\varphi$ satisfies weak synchronization. Then $\varphi$ is strongly mixing and for all $x,y\in E$ 
       $$d(\varphi_{t}(\cdot,x),\varphi_{t}(\cdot,y)) \to 0 \text{ for }t\to\infty,$$
    in probability. In particular, $\varphi$ satisfies weak global pointwise stability.
\end{rem}
\begin{proof}
  Let $A(\omega)=\{a(\omega)\}$ $\PP$-a.s.\ be the weak point attractor of $\varphi$. Then, $\rho = \E \d_{a(\omega)}$ is an invariant probability measure for $(P_t)$. 
  By Lemma \ref{lem:strong_mix}, $\varphi$ is strongly mixing. Moreover,
     \begin{align*}
        d(\varphi_{t}(\cdot,x),\varphi_{t}(\cdot,y)) 
        \le d(\varphi_{t}(\cdot,x),a(\t_t\omega)) + d(\varphi_{t}(\cdot,y),a(\t_t\omega)) \to 0\text{ for }t\to\infty,
     \end{align*}  
     in probability.
\end{proof}

\subsection{Case of distance being a diffusion}

In this section we consider the special case in which the distance process of any pair of trajectories is a diffusion. The set-up is the same as in the 
previous subsection. In addition, we assume that for each pair of initial points $x,y \in E$ the distance process 
$r_t:=d(\varphi_t(x),\varphi_t(y)),\,t \ge 0$ satisfies a scalar stochastic differential equation
\begin{equation}\label{rho}
 \dd r_t= b(r_t) \,\dd t + \sigma(r_t)\, \dd W_t,
\end{equation}
on $I:=(0,R)$ for some $R\in (0,\infty]$, where $W$ is standard Brownian motion. 
We assume that $b$ and $\sigma$ are continuous on $I$, $\sigma >0$ on $I$, and $\lim_{x \downarrow 0}\sigma(x)=0$. Recall the following definition of the {\em scale function} 
$s$ and the {\em speed measure} $m$ on $I$ (see e.g.~\cite{K02} or \cite{KS91}):
\begin{align*}
 s(x):=&\int_c^x \exp\Big\{ \int_c^y \frac{2b(z)}{\sigma^2(z)}\,\dd z\Big\}\,\dd y,\\
 m(\dd x):=&\frac 2 {s'(x)\sigma^2(x)}\,\dd x,
\end{align*}
for $x \in I$ and an arbitrary point $c \in I$. Assume that none of the boundary points of $I$ are accessible (which can be expressed in terms of $b$ and $\sigma$ 
via Feller's test for explosions, see \cite[p.342ff]{KS91}) and that the speed measure $m$ is finite away from 0, i.e.\ $m[\varepsilon,R)<\infty$ for some (and hence for all) 
$\varepsilon \in I$.

\begin{lemma}\label{mainlemma}
Under the assumptions above we have the following equivalences
\begin{itemize}
  \item[(i)] $r_t \to 0$ in probability 
  \item[(ii)]$m(I)$  is infinite. 
\end{itemize}
\end{lemma}
\begin{proof}
First note that the fact that $m$ is finite away from 0 and the boundary $R$ is inaccessible imply $s(R)=\infty$. 

 (i) $\Rightarrow$ (ii): if $m$ is finite, then the diffusion $r$ is ergodic with invariant probability measure $m/m(I)$ by \cite[Theorem 23.15]{K02} contradicting (i).
 
 (ii) $\Rightarrow$ (i): this follows from \cite[Proposition 2.1(iii)]{S02}.
\end{proof}

\section{Examples}

Let $\varphi$ be a white noise RDS on a complete, $d$-dimensional smooth Riemannian manifold $E$ with respect to an ergodic metric dynamical system 
$(\Omega,\PP,\theta)$ and let $(P_t)$ be the associated Markovian semigroup. Further assume that $\varphi_t(\omega,\cdot)\in C^{1,\delta}_{loc}$ for some 
$\delta >0 $ and all $t\ge 0$ and that $P_1$ has an ergodic invariant measure $\rho$ such that
  $$\E\int_E \|D\varphi_1(\omega,x)\|\,\dd \rho(x)<\infty.$$
Following \cite{C86} for the case of $E$ being compact and \cite{FGS14} for the case $E=\R^d$ this implies the existence of constants 
$\l_m <\dots < \l_1,\,m\le d$, the so-called Lyapunov spectrum, such that 
   $$\lim_{n\to\infty} \frac{1}{n} \log\|D\varphi_n(\omega,x)v\|\in \{\l_i\}_{i=1}^m, $$
for $\PP\otimes\rho$-a.a.\ $(\omega,x)\in \Omega\times M$ and all $v\in T_x M \setminus \{0\}$. We define the top Lyapunov exponent by $\l_{top}=\l_1$.

\subsection{Isotropic Brownian flows on the sphere}

We study isotropic Brownian flows on $E=S^{d-1}:=\{x \in \R^d:|x|=1\}$ for $d \ge 3$, where $|\cdot|$ denotes the Euclidean norm. 
On $E$, we define the metric $d(x,y)=\arccos\langle x ,y \rangle$, where $\langle.,.\rangle$ denotes the standard 
inner product on $\R^d$. Note that the metric $d$ takes values in $[0,\pi]$. The reason for 
excluding the case $d=2$ (i.e.~$E=S^1$) is that in that case the boundary $\pi$ of the two-point distance is accessible and therefore requires 
a slightly different treatment. 
Given smooth vector fields $V_j,\, j\in\{1,2,...\}$ and $V$ on the sphere $S^{d-1}$ and $iid$ standard Brownian motions $B^j$, a stochastic flow on $S^{d-1}$ is defined by the solution of the Stratonovich equation
\begin{eqnarray}\label{flowdef}
\begin{split}
\dd\varphi_t(x)=&\partial W_{\varphi_t(x)}(t),\\
\varphi_0(x)=&x
\end{split}
\end{eqnarray}
where the vector field valued semi-martingale $W_x(t)$ is defined as
\begin{eqnarray}\label{Wdef}
W_x(t)=\sum_{j=1}^\infty V_j(x)B^j_t+V(x)t.
\end{eqnarray}
The flow $\varphi$ is completely determined by its characteristics $a:S^{d-1}\times S^{d-1}\to TS^{d-1}\times TS^{d-1}$ and $b:S^{d-1}\to TS^{d-1}$ which are given by
\begin{eqnarray*}\label{chardef}
\begin{split}
\left<a(x,y),\dd f_x\otimes \dd g_y\right>=&\lim_{t\searrow0}\frac{1}{t}E[(f(\varphi_t(x))-f(x))(g(\varphi_t(y))-g(y))]\\
\left<b(x),\dd f_x\right>=&\lim_{t\searrow0}\frac{1}{t}E[f(\varphi_t(x))-f(x)].
\end{split}
\end{eqnarray*}
In terms of the vector fields used in (\ref{Wdef}) one has
\begin{align*}
a(x,y)=&\sum_{j=1}^\infty V_j(x)\otimes V_j(y),\\
b(x)=&V+\sum_{j=1}^\infty \nabla_{V_j}V_j,
\end{align*}
where $\nabla$ denotes the Riemannian connection on $S^{d-1}$ compatible with the standard metric defined above on the sphere. The flow is called Brownian if the one point motion is a Brownian motion on $S^{d-1}$ and this corresponds to the requirement
$\left<b(x),\dd f_x\right>=\frac12 \Delta f(x)$
where $\Delta$ is the Laplace-Beltrami operator on $S^{d-1}.$ The definition of isotropy for a flow on a sphere relies on the fact that the sphere is a homogeneous space. In the present context, the space $S^{d-1}=\mathcal{O}(d-1)/\mathcal{O}(d-2)$ where $\mathcal{O}(d-1)$ is the orthogonal group. The group $\mathcal{O}(d-1)$ acts transitively on $S^{d-1}$ and the flow $\varphi$ is called isotropic if for all $g\in \mathcal{O}(d-1),$
\begin{eqnarray}
\{g^{-1}\varphi_t(gx):x\in S^{d-1},\,t\ge 0\}\stackrel{\mathcal{L}}{=}\{\varphi_t(x):x\in S^{d-1},\,t\ge 0\}.
\end{eqnarray}
In terms of the characteristics, the flow is isotropic if  and only if for every $g\in \mathcal{O}(d-1)$
\begin{eqnarray*}
\begin{split}
\left<a(gx,gy),\dd g_x(u)\otimes \dd g_y(v)\right>=&\left<a(x,y),u\otimes v\right>, \,x,y\in S^{d-1},\,(u,v)\in T_xS^{d-1}\otimes T_yS^{d-1}\\
\left<b(gx),dg_x(u)\right>=&\left<b(x),u\right> ,\,x\in S^{d-1},\,u\in T_xS^{d-1}.
\end{split}
\end{eqnarray*}
Here $\dd g_x$ means the differential of the mapping $g:S^{d-1}\to S^{d-1}$ at the point $x\in S^{d-1}.$ It was shown in \cite{LJW84} that the 
solution of (\ref{flowdef}) defines a flow of diffeomorphisms of $S^{d-1}$ into itself and satisfies the definition of an RDS.

\begin{thm}
   Let $\varphi$ be an isotropic Brownian flow on $S^{d-1}$ satisfying $\l_1 \le 0$. Then weak synchronization holds.
\end{thm} 
\begin{proof}
We recall from \cite[Theorem 4.1]{R99} that $r_t := d(\varphi_t(\cdot,x),\varphi_t(\cdot,x))$ is a diffusion and 
  $$\dd r_t = b(r_t)\dd t + \sigma(r_t) \dd W_t, $$
where
\begin{align*}
  \sigma^2(r) &= 2[\a(1)-\a(\cos r)\cos r+\b(\cos r)\sin^2 r]\\
  b(r) &= \frac{d-2}{\sin r} (\a(1)\cos r-\a(\cos r)),
\end{align*}
with  
\[\alpha(r)=\sum_{l=1}^\infty a_l \gamma_l(r)+\sum_{l=1}b_l\left(r\gamma_l(r)-\frac{1-r^2}{d-2}\gamma'_l(r)\right)\]
\[\beta(r)=\sum_{l=1}^\infty a_l \gamma'_l(r)+\sum_{l=1}b_l\left(-\gamma_l(r)-\frac{r}{d-2}\gamma'_l(r)\right)\]
and
 \[\gamma_l(r)=\frac{C^{d/2}_{l-1}(r)}{C^{d/2}_{l-1}(1)}\]
where $C^{d/2}_{l-1}$ are the Gegenbauer polynomials. The sequences $\{a_l\}$ and $\{b_l\}$ are summable with nonnegative terms.
Moreover, from the proof of \cite[Theorem 4.1]{R99} we recall that, since $d\ge 3$,
  $$\PP[\inf\{t> 0:r_t = \pi\}=\infty]=1.$$
In \cite[Theorem 3.1]{R99} the Lyapunov spectrum $(\l_1,\dots,\l_{d-1})$ is shown to be given by
  $$\l_n = \frac{d-2n-1}{2}\a'(1)-\frac{d-1}{2}\a(1)-n\b(1)\quad n=1,\dots,d-1.$$
In particular, $\l_1= \frac{d-3}{2}\a'(1)-\frac{d-1}{2}\a(1)-\b(1)$.

Moreover, if $\l_1 <0$ then $\lim_{t\to\infty}r_t = 0$ $\PP$-a.s. and if $\l_1 \ge 0$ then $\PP[\limsup_{t\to\infty}r_t = \pi] = \PP[\liminf_{t\to\infty}r_t = 0]=1$.

Let $s$ be the scale function and $m$ the speed measure  associated to the diffusion $r_t.$ From \cite[proof of Theorem 4.1]{R99} we recall, close to $0$,
\begin{align*}
  b(x) &= x \left(\frac{d-2}{2}\right)(\a'(1)-\a(1))+o(|x|^2)\\
  \s^2(x) &= x^2(\a'(1)+\a(1)+2\b(1))+o(|x|^3)\\
  s'(x)&=\exp\Big\{-\int_\varphi^x \frac{2b(\t)}{\s^2(\t)} \dd\t \Big\} \approx c_1 x^{-\g_1} 
\end{align*}
with $\g_1=\frac{(d-2)(\a'(1)-\a(1))}{\a'(1)+\a(1)+2\b(1)}$, where $c_1$ is a positive constant and $\approx$ means the ratio of the two sides tends to $1$ as $x$ tends to $0.$ On the other hand, close to $\pi$,
\begin{align*}
  \s^2(x) &= 2(\a(1)+\a(-1)) + (\a(-1)-\a(-1)+2\b(-1))(x-\pi)^2+o(|x-\pi|^2)\\
  \frac{2b(x)}{\s^2(x)} &= -\frac{d-2}{\pi-x}+o(1) \\
  s'(x)&=\exp\Big\{-\int_\varphi^x \frac{2b(\t)}{\s^2(\t)}\dd \t \Big\} \approx c_2 (\pi - x)^{-(d-2)}.
\end{align*}
We note that $\g_1 \ge 1$ iff $\l_1 \ge 0$ which is seen by noting $\lambda_1\ge0$ is the same as $\frac{d-3}{2}\a'(1)-\frac{d-1}{2}\a(1)\ge\b(1).$ Thus, if $\lambda_1\ge0,$
\begin{align*}
\g_1&=\frac{(d-2)(\a'(1)-\a(1))}{\a'(1)+\a(1)+2\b(1)}\\
&\ge\frac{(d-2)(\a'(1)-\a(1))}{\a'(1)+\a(1)+2(\frac{d-3}{2}\a'(1)-\frac{d-1}{2}\a(1))}\\
&=1
\end{align*}
and the other equivalence is proven in a similar fashion.
 Hence,  $\l_1 \ge 0$ iff $s(0+)=-\infty$. Furthermore, $s(\pi-) = \infty$ iff $d\ge 3$. 

We further note that
\begin{align*}
m((0,\pi))
   &=\int_{(0,\pi)} \frac{2}{\s^2(x) s'(x)}\dd x \\
   &=\int_{(0,\frac{\pi}{2})} \frac{2}{\s^2(x) s'(x)}\dd x + \int_{(\frac{\pi}{2},\pi)} \frac{2}{\s^2(x) s'(x)}\dd x.
\end{align*} 
For $\e>0$ small we have
\begin{align*}
   \int_{(0,\e)} \frac{1}{\s^2(x) s'(x)}\dd x &\approx c_3 \int_{(0,\e)} x^{\g_1-2} \dd x =\begin{cases}
      \infty & \text{if } \g_1 \le 1\\
      <\infty & \text{if } \g_1 > 1.
   \end{cases}
\end{align*} 
and 
\begin{align*}
   \int_{(\pi-\e,\pi)} \frac{1}{\s^2(x) s'(x)}\dd x&\approx c_4\int_{(\pi-\e,\pi)}(\pi - x)^{d-2}\dd x
    <\infty.
\end{align*} 
Hence, $m((0,\pi))=\infty$ iff $\g_1\le 1$ iff $\l_1 \le 0$ and the claim follows from Lemma \ref{mainlemma} and 
Corollary \ref{cor:main}. 
\end{proof}
\begin{rk}
 Note that {\em synchronization} can never occur for an IBF on $S^{d-1}$ since $\varphi_t(\omega,\cdot): S^{d-1}\to S^{d-1}$ is a homeomorphism 
 almost surely and $S^{d-1}$ is compact.
\end{rk}

\subsection{Isotropic Ornstein-Uhlenbeck flows}
Next, we consider isotropic Ornstein-Uhlenbeck flows (IOUF) on $\R^d$. Consider a stochastic flow generated by the following  
{\em Kunita-type SDE} on $\R^d$ 
\begin{equation}\label{IOUF}
\dd X_t=-cX_t\,\dd t + M(\dd t,X_t),
\end{equation}
where $c\ge 0$ and $M$ is a (normalized) isotropic Brownian field, i.e.
\begin{enumerate}
 \item $M:[0,\infty)\times \R^d \times \Omega \to \R^d$ is a centered Gaussian process which is continuous in the first two variables.
 \item There exists a {\em covariance tensor} $b:\R^d \to \R^{d \times d}$ such that 
   $${\mathrm{cov}}(M_i(t,x),M_j(s,y))=(s\wedge t) b_{i,j}(x-y),\ s,t \ge 0,\  x,y \in \R^d.$$
 \item $b$ is a $C^4$ function with bounded derivatives up to order 4, $b(0)=I_d$ is the identity matrix and $x \mapsto b(x)$ is not constant.
 \item $b$ is {\em isotropic}, i.e. $b(x)=O^T b(Ox)O$ for each orthogonal matrix $O$ and all $x \in \R^d$. 
\end{enumerate}
Note that these properties imply that $t \mapsto M(t,x)$ is a $d$-dimensional standard Brownian motion for each fixed $x \in \R^d$. For the definition of the solution 
of an equation like \eqref{IOUF} and the fact that the equation generates a {\em stochastic flow} $\varphi$, see \cite{Ku90}. For a proof of the fact that \eqref{IOUF} 
generates an RDS $\varphi$ (which has a random attractor when $c>0$), see \cite{DS11}.  We will use the term {\em isotropic Brownian flow} (IBF) for both $\phi$ and $\varphi$ 
if $c=0$ and {\em isotropic Ornstein-Uhlenbeck flow} if $c>0$.

IBFs have been studied by several authors, among them \cite{BH86} who also provide a formula for its Lyapunov spectrum. Before stating the formula for the top exponent 
we mention that the covariance tensor $b$ can be represented in the form
$$
b_{i,j}(x)=(B_L(|x|)-B_N(|x|))\frac{x_ix_j}{|x|^2} +\delta_{i,j}B_N(|x|),\;x \in \R^d\backslash \{0\},\;i,j=1,\cdots,d,
$$
where $B_L$ and $B_N$ (the {\em longitudinal} respectively {\em transversal} or {\em normal} covariance functions) are defined as 
$B_L(r):=b_{p,p}(r e_p)$ and $B_N(r):=b_{p,p}(r e_q)$, where $r \ge 0$, $p\neq q$ and $e_i$ is the $i$-th standard basis vector in 
$\R^d$. Abbreviating $\beta_L:=-B_L''(0)$ and $\beta_N:=-B_N''(0)$ the formula 
for the top exponent $\lambda_1$ reads
$$
\lambda_1= (d-1)\frac {\beta_N}2-\frac{\beta_L}2
$$
in case $c=0$ (note that due to translation invariance a Lyapunov spectrum can be defined in spite of the fact that 
the associated Markov process (which is $d$-dimensional Brownian motion) does not have an invariant probability measure). 
It is easy to see (and was proved in \cite{D06}) that for $c>0$ the Lyapunov spectrum just shifts by $c$. In particular,
$$
\lambda_1= (d-1)\frac {\beta_N}2-\frac{\beta_L}2 -c.
$$
The two-point distance $(r_t)$ is a diffusion on $(0,\infty)$ satisfying the equation
$$
\dd r_t= \big((d-1)\frac{1-B_N(r)}r -c r\big)\,\dd t + \sqrt{2\big( 1-B_L(r)\big)}\,\dd W_t
$$
(this is shown in \cite{BH86} for $c=0$ and in \cite{D06} for $c>0$). Note that an IBF and an IOUF is a white-noise RDS. While the Markov 
process associated to an IBF is Brownian motion which does not have an invariant probability measure the Markov process associated to an IOUF 
is a $d$-dimensional Ornstein-Uhlenbeck process which 
has an invariant probability measure $\rho$ (a centered Gaussian measure). 
According to \cite[Proposition 7.1.3]{D06}, the speed measure of $r$ is finite iff $\lambda_1>0$ (this fact is not difficult to check given the SDE for the 
two-point distance). We conclude

\begin{thm}
   Let $\varphi$ be an isotropic Ornstein-Uhlenbeck flow on $\R^d$ satisfying $\lambda_1 \le 0$. Then weak synchronization holds.
\end{thm}

\bibliographystyle{plain}
\bibliography{refs}

\end{document}